\documentclass[10pt]{amsart}
\usepackage{amssymb,amsfonts,lineno,amsthm,amscd,stmaryrd}
\usepackage[leqno]{amsmath}
\usepackage[numbers]{natbib}

\theoremstyle{plain}
\newtheorem{thm}{Theorem}

\newtheorem{lem}[thm]{Lemma}

\theoremstyle{definition}

\newtheorem{remark}[thm]{Remark}

\numberwithin{thm}{section}
\numberwithin{equation}{section}

\newcommand{\EQ}[1]{\eqref{#1}}

\newcommand{\THM}[1]{Theorem~\ref{#1}}

\newcounter{hypo}

\DeclareMathOperator{\trace}{trace}

\DeclareMathOperator{\divg}{div}

\newcommand{\R}{\ensuremath{\mathbb{R}}}
\newcommand{\rn}{\R^n}

\newcommand{\iden}{\ensuremath{I_n}}
\newcommand{\ep}{\varepsilon}
\newcommand{\Sy}{\ensuremath{\mathcal{S}_n}}
\newcommand{\puccisub}[2]{\mathcal{P}^-_{#1,#2}}
\newcommand{\Puccisub}[2]{\mathcal{P}^+_{#1,#2}}
\newcommand{\PucciSub}{\Puccisub{\elp}{\Elp}}
\newcommand{\pucciSub}{\puccisub{\elp}{\Elp}}
\newcommand{\pucci}{\pucciSub}
\newcommand{\Pucci}{\PucciSub}

\newcommand{\Elp}{\Lambda}
\newcommand{\elp}{\lambda}
\newcommand{\anom}{{\alpha^*}}

\begin{document}
\title[Sharp {L}iouville results for fully nonlinear equations]{Sharp {L}iouville results for fully nonlinear equations with power-growth nonlinearities}
\author{Scott N. Armstrong}
\address{Department of Mathematics\\ Louisiana State University\\ Baton Rouge, LA 70803.}
\email{armstrong@math.lsu.edu}
\author{Boyan Sirakov}
\address{UFR SEGMI, Universit\'e Paris 10\\
92001 Nanterre Cedex, France \\
and CAMS, EHESS \\
54 bd Raspail \\
75270 Paris Cedex 06, France}
\email{sirakov@ehess.fr}
\date{\today}

\begin{abstract}
We study fully nonlinear elliptic equations such as
\[
F(D^2u) = u^p, \quad p>1,
\]
in $\R^n$ or in exterior domains, where $F$ is any uniformly elliptic, positively homogeneous operator. We show that there exists a critical exponent, depending on the homogeneity of the fundamental solution of $F$, that sharply characterizes the range of $p>1$ for which there exist positive supersolutions or solutions in any exterior domain. Our result generalizes theorems of Bidaut-V\'eron \cite{B} as well as Cutri and Leoni \cite{CL}, who found critical exponents for supersolutions  in the whole space $\R^n$, in case $-F$ is Laplace's operator and Pucci's operator, respectively. The arguments we present are new and rely only on the scaling properties of the equation and the maximum principle.
\end{abstract}

\maketitle

\section{Introduction and main results}

Elliptic equations and systems with power-like zero order terms have been the focus of great attention for many years. A model is the \emph{Emden-Fowler equation}
\begin{equation} \label{lapl}
-\Delta u = u^p, \quad p > 1,
\end{equation}
which has important applications in physics and geometry as well as a rich mathematical structure. In this paper, we study the more general equation
\begin{equation} \label{pdiegen}
F(D^2u) = f(x,u)
\end{equation}
where $F$ is a uniformly elliptic  operator and $f$ has power-like dependence in $u$. In particular, we study the existence of positive solutions and supersolutions of \EQ{pdiegen} in unbounded domains of $\R^n$, $n\ge 2$. The only hypotheses on the nonlinear operator $F$ are uniform ellipticity and positive homogeneity; precisely, we require:
\begin{itemize}
 \item[(H1)]  for some constants $0<\lambda\le \Lambda$ and all real symmetric matrices $M$ and $N$, with $N$ nonnegative definite, we have
    \begin{equation*}
    \lambda\trace(N) \le F(M-N)-F(M) \leq \Lambda\trace(N), \quad \mbox{and}
    \end{equation*}
\item[(H2)]  $F(tM)= tF(M)$ for every $t\geq 0$ and each symmetric matrix $M$.
\end{itemize}
An operator $F$ satisfying (H1)-(H2) is called an \emph{Isaacs operator}. If in addition $F$ is concave or convex, then $F$ is often called a \emph{Hamilton-Jacobi-Bellman} operator. See the references in the next section for more on the the theory and applications of Isaacs operators.

\smallskip

The question we are concerned with has been extensively studied in the special case $F(D^2u) = -\Delta u$ (when (H1) holds with $\lambda=\Lambda=1$, and (H2) holds for all $t\in \R$), by using energy methods and the divergence-form structure of the Laplacian. The following Liouville-type theorem is a particular case of results obtained by Bidaut-V\'eron \cite{B}, and sharply characterizes the range of  $p>1$ for which  there exist positive (super)solutions of \EQ{lapl} in exterior domains, as well as positive supersolutions in the whole space $\R^n$. An \emph{exterior domain} is a domain $\Omega \subset \R^n$ for which $\rn\setminus B_{R}\subseteq\Omega\subseteq \rn\setminus \{0\}$ for some $R>0$, where $B_R$ denotes the open ball of radius $R$ centered at the origin.
\begin{thm}[\cite{B}] \label{BEB}
Denote $2_* := n / (n-2)$ if $n\geq 3$, and $2_* :=\infty$ if $n=2$. Let $p>1$. Then the Emden-Fowler equation \EQ{lapl} has no nontrivial nonnegative weak supersolutions in any exterior domain of $\R^n$, provided that $p \leq 2_*$.
\end{thm}

Note that it can be checked by a straightforward calculation the functions $v(x):= c_p|x|^{-2/(p-1)}$ and $u(x) := \tilde c_p \left( 1 + |x|^2 \right)^{-1/(p-1)}$ are respectively a solution of \EQ{lapl} in $\R^n \setminus \{ 0 \}$ and a supersolution of \EQ{lapl} in the whole space $\R^n$, provided that $p > 2_*$ and the constants $c_p, \tilde c_p>0$ are chosen appropriately. Thus the previous theorem states that $2_*$ is the \emph{critical exponent} for both the existence of positive solutions or supersolutions of \EQ{lapl} in any exterior domain, and positive supersolutions of \EQ{lapl} in the whole space $\R^n$. Recall that the famous result of Gidas and Spruck \cite{GS} states that the critical exponent for existence of positive solutions of \EQ{lapl} in $\R^n$ is different, namely $2^*:=(n+2)/(n-2)$. The delicate proof of this deep fact relies on special geometric properties and symmetries of the Laplacian.

Theorem \ref{BEB} has been extended to various quasilinear equations in divergence form, like
$
-\Delta_m u = u^p,
$ see \cite{B},
as well as to more general divergence form equations such as
$-\divg(A(Du,u,x)) = 0$,
provided that $A$ possesses the appropriate growth in $u$ and $Du$. For more results in this direction, see \cite{GV,MP,BVP,SZ}.

\medskip

In this paper, we study the corresponding problem for certain fully nonlinear elliptic equations. Consequently, all differential (in)equations appearing here are to be interpreted in the viscosity sense, which is the appropriate notion of weak solution for elliptic equations in nondivergence form. Our model equation is
\begin{equation} \label{model}
F(D^2u) = u^p,
\end{equation}
where $p > 1$ and $F$ satisfies (H1) and (H2). The arguments we give  are completely different from the ones in the above quoted works, since energy methods are obviously inapplicable to the study of \EQ{model}. Instead, our approach  relies on the maximum principle, and makes essential use of the following result we recently obtained with Smart \cite{ASS} on the existence and properties of fundamental solutions of fully nonlinear  equations.

\begin{thm}[\cite{ASS}]\label{fundysol} Assume $F$ satisfies (H1) and (H2).
 The equation
 $F(D^2u)=0$ in $\R^n \setminus \{0\}$
  has a non-constant solution
that is bounded below in $B_1$ and bounded above in $\R^n \setminus B_1$. Moreover, the set of all such solutions is  $\{a\Phi + b\;|\; a>0, b\in \R\}$, where  $\Phi\in C^{1,\delta}_{\mathrm{loc}}(\R^n \setminus \{ 0 \})$ can be chosen to satisfy one of the following homogeneity relations: for all $\sigma>0$,
\begin{equation}\label{scale}
{\Phi}(x) = {\Phi}(\sigma x) + \log \sigma\qquad\mbox{or}\qquad
\left\{\begin{array}{l} {\Phi}(x) =\sigma^{\anom} {\Phi}(\sigma x)\\
\anom\Phi(x)>0\end{array}\right.\quad\mbox{ for }\;  x\in \R^n \setminus \{0\},
\end{equation}
for some unique number $\anom=\anom(F) \in (-1,\infty)\setminus\{0\}$ which depends only on $F$ and $n$ (we set $\anom(F)=0$ in case the first alterantive in \EQ{scale} occurs). We call $\alpha^*(F)$ the \emph{scaling exponent} of $F$, and we call $\Phi$ the \emph{fundamental\ solution} of $F$.
\end{thm}
\begin{remark} Of course, the fundamental solution of the Laplacian  is $\Phi(x) = |x|^{2-n}$, if $n\ge3$, and $\Phi(x)= -\log|x|$, if $n=2$, so the scaling exponent of $-\Delta$ is $\alpha^*(-\Delta) = n-2$. For a general $F$ satisfying (H1) and (H2), the scaling exponent can be any number between
$
\frac{\lambda}{\Lambda}(n-1) - 1$ and $ \frac{\Lambda}{\lambda}(n-1) - 1.$ For more on fundamental solutions and scaling exponents, we refer to \cite{ASS}.
\end{remark}

Equation \EQ{model} has another ``scaling exponent",  given by the interplay between the $1$-homogeneity of the elliptic operator and the $p$-homogeneity of the right-hand side in \EQ{model}.  If $u$ is a (sub/super)solution of \EQ{model} in the exterior domain $\R^n \setminus B_R$, then it is easy to check that the rescaled function $u_\sigma$, defined for each $\sigma > 0$ by
\begin{equation*}
u_\sigma (x):= \sigma^{\beta^*} u(\sigma x), \qquad \beta^*=\beta^*(p) : =\frac{2}{p-1}\,,
\end{equation*}
 is also a (sub/super)solution of \EQ{model} in the domain $\R^n \setminus B_{R/\sigma}$. In order to compare with \EQ{scale}, note that $\Phi_\sigma\equiv\Phi$ when $\alpha^*=\beta^*$.

\medskip

Put briefly, our  main result for \EQ{model} asserts that the answer to the question of whether there exists a positive (super)solution in exterior domains, or a positive supersolution in the whole space $\R^n$, is determined by \emph{which of the numbers $\alpha^*(F)$ and $\beta^*(p)$ is greater.}

\begin{thm} \label{theomodel}
Assume that $F$ satisfies (H1) and (H2), and $p > 1$. Then the generalized Emden-Fowler equation \EQ{model}
\begin{enumerate}
\item has no nontrivial nonnegative supersolution in any exterior domain of~$\R^n$, if $\alpha^*(F) \le \beta^*(p)$;
\item has a positive supersolution in the whole $\R^n$, if  $\alpha^*(F) > \beta^*(p)$.
\item has a positive solution in $\R^n \setminus \{0\}$, if  $\alpha^*(F) > \beta^*(p)$.
\end{enumerate}
\end{thm}

This result extends Theorem \ref{BEB} to arbitrary Isaacs operators,  and provides a new perspective on this theorem by embedding the Laplacian in the family of positively homogeneous, uniformly elliptic operators. From our point of view, the condition $p\le n/(n-2)$ is better written as $n-2\leq 2/(p-1)$, which emphasizes the competition between the homogeneity of the fundamental solution of the Laplacian and the scaling exponent $\beta^*(p) = 2/(p-1)$ for equation \EQ{lapl}. Of course, we may also write the inequality $\alpha^*(F) \leq \beta^*(p)$ in terms of $p$ as
\begin{equation} \label{critex}
p \leq 2_*(F):= \begin{cases}
\frac{\alpha^*(F)+2}{\alpha^*(F)} & \mbox{if} \quad \alpha^*(F) > 0,\\
\infty & \mbox{if} \quad \alpha^*(F) \leq 0.
\end{cases}
\end{equation}

\medskip

To our knowledge, there are no previous results concerning the nonexistence of positive solutions of fully nonlinear equations and inequalities in exterior domains. Theorem \ref{theomodel} implies \EQ{model} has no positive {\it singular} supersolutions, with singularities contained in a bounded set, provided that $ \alpha^*\le \beta^*$.

In the special case that \EQ{model} is posed in the whole space $\R^n$, and $F$ is a Pucci extremal operator (see Section \ref{preliminaries} for a definition of these operators), problem \EQ{model} was studied and Theorem \ref{theomodel} was proved by Cutri and Leoni \cite{CL}. Their result was extended to a class of rotationally invariant operators\footnote{\emph{rotationally invariant} means that $F(M)$ depends only on the eigenvalues of $M$.} in a recent paper by Felmer and Quaas \cite{FQ}. Note that for rotationally invariant operators the existence statements (ii) and (iii) in Theorem \ref{theomodel} are trivial. For such operators, just as for \EQ{lapl}, it can be checked by a direct computation that if $\alpha^* > \beta^*$, then the functions $v(x):= c|x|^{-2/(p-1)}$ and $u(x) := \tilde c \left( 1 + |x|^2 \right)^{-1/(p-1)}$ are respectively a solution of \EQ{model} in $\R^n \setminus \{ 0 \}$ and a supersolution of \EQ{model} in the whole space $\R^n$, for suitably chosen constants $\tilde c, c > 0$. The proofs of the  nonexistence statements in \cite{CL} and \cite{FQ} depend heavily on the rotational invariance of $F$, in particular on the existence of radial fundamental solutions for $F$.

\smallskip

Theorem~\ref{theomodel} generalizes these results first by dropping the assumption of rotational invariance and considering an arbitrary fully nonlinear operator satisfying (H1) and (H2), and second by requiring only that the equation hold in an exterior domain, which yields a Liouville statement for singular solutions. Our proof of Theorem \ref{theomodel} is (necessarily) accomplished through a different argument than the one given in both \cite{CL} and \cite{FQ}. The proof of part (i) makes use of the scaling properties of the equation and of the maximum principle to compare a positive supersolution of \EQ{model} with the fundamental solution $\Phi$ of $F$. To prove part (ii), we show that the ellipticity estimates imply that some power of the fundamental solution of $F$ is a supersolution of \EQ{model}, which permits us to construct a solution of \EQ{model} in the whole space by a truncation-type argument.

Finally, Theorem \ref{theomodel}(iii) is proved with the help of Krasnoselskii degree theory (see Theorem \ref{fixedpt} in Section \ref{proofs}), a tool which became popular in elliptic PDE with the well-known works of Amann (see e.g. \cite{Am}). An essential step in the application of degree theoretic results lies in estabishing {\it a priori} bounds. The classical blow-up argument, introduced by Gidas and Spruck \cite{GS2} and de Figueiredo, Lions and Nussbaum \cite{DLN}, has been employed many times in the last thirty years to obtain a priori bounds and hence existence results for nonlinear problems in bounded domains. In the fully nonlinear setting this approach was first used in \cite{QS1}.

It is less typical for a priori estimates and degree theory to be used to establish existence results in \emph{unbounded domains}, and their application hides some specificities. In particular, we deduce an a priori estimate by an argument which is different than the one typically used in the literature (for example in \cite{DLN,GS2}). In fact, we will prove the following result, which implies Theorem \ref{theomodel}(iii) and deserves to be stated separately.

\begin{thm}\label{theoexist} Assume $\beta^*(p) < \alpha^*(F)$. Then at least one of the following holds:
\begin{enumerate}
\item equation \EQ{model} has a bounded positive solution in the whole space $\R^n$, or
\item equation \EQ{model} has a positive solution $u$ in the domain $\R^n \setminus \{ 0 \}$ which is $\left( -\beta^*(p)\right)$-homogeneous, that is, $u_\sigma \equiv u$ for all $\sigma > 0$.
\end{enumerate}
\end{thm}

In Section \ref{proofs} we give an extended discussion of our arguments and compare them to those of previous papers on the subject. Let us mention here another advantage of our approach, which is that it easily adapts to more general nonlinearities $f = f(x,u)$ as in \EQ{pdiegen}. An extension of our results to equation \EQ{pdiegen} is given in Theorem~\ref{mainthm}, which roughly states that analogous conclusions as in Theorem \ref{theomodel} (i)-(ii) hold for \EQ{pdiegen}, provided $f$ is positive, behaves like $|x|^{-\gamma}u^p$ for large $|x|$ and small $u$, and satisfies a mild global hypothesis.

It remains an interesting open question whether the result of Gidas and Spruck \cite{GS} can be extended to fully nonlinear equations. For more details on this question, we refer to Felmer and Quaas \cite{FQ2}, who used ODE techniques to study the existence and nonexistence of radial solutions in $\R^n$ for rotationally invariant operators.

In the next section, we very briefly describe some notation and results from the theory of viscosity solutions of fully nonlinear equations. Our main result is proved in Section \ref{proofs}.

\section{Some notations and results on viscosity solutions of elliptic PDE}\label{preliminaries}

This short section is meant for readers who are not familiar with the theory of fully nonlinear equations and the concept of viscosity solutions.

We denote the set of $n$-by-$n$ real symmetric matrices  by $\Sy$, and $\iden\in\Sy$ is the identity matrix.  If $x,y \in \R^n$, we denote by $x\otimes y$ the symmetric matrix with entries $\frac{1}{2} ( x_iy_j + x_jy_i)$. For $M,N \in \Sy$, we write $M\geq N$ if $M-N$ has nonnegative eigenvalues. The Pucci extremal operators are defined  by
\begin{equation*}
\PucciSub(M) := \sup_{A\in \llbracket\elp,\Elp\rrbracket} \left[ - \trace(AM) \right] \quad \mbox{and} \quad \pucciSub (M) := \inf_{A\in \llbracket\elp,\Elp\rrbracket} \left[ - \trace(AM) \right],
\end{equation*}
for each $M\in\Sy$ and $0 < \lambda \leq \Lambda$, where $\llbracket\lambda,\Lambda\rrbracket \subseteq \Sy$ is the subset of $\Sy$ consisting of the matrices $A$ for which $\lambda \iden \leq A \leq \Lambda \iden$.  The following equivalent definition of the Pucci extremal operators is often more convenient for calculations:
\begin{equation}\label{puccinice}
\Pucci(M) = -\lambda \sum_{\mu_j > 0} \mu_j - \Lambda \sum_{\mu_j < 0} \mu_j \quad \mbox{and} \quad \pucci(M) = -\Lambda \sum_{\mu_j > 0} \mu_j - \lambda \sum_{\mu_j < 0} \mu_j,
\end{equation}
where $\mu_1, \ldots, \mu_n$ are the eigenvalues of $M$. See Caffarelli and Cabre \cite{CC} for more on these operators (to avoid confusion with the notations in \cite{CC}, note that $\Pucci(M) = -{\mathcal{M}}^-(M,\lambda,\Lambda)$ and $\pucci(M) = -{\mathcal{M}}^+(M,\lambda,\Lambda)$, where $\mathcal{M}^\pm$ are as in \cite{CC}).

An equivalent way of writing (H1) is
\begin{itemize}
\item[(H1)] there exist $0 < \lambda \leq \Lambda$ such that for every $M,N \in \Sy$, \[ \puccisub{\lambda}{\Lambda}(M-N) \leq F(M) - F(N) \leq \Puccisub{\lambda}{\Lambda}(M-N).\]
\end{itemize}
Observe that (H1) and (H2) are satisfied for both $F=\pucci$ and $F=\Pucci$, and these hypotheses imply $\pucci(M) \leq F(M) \leq \Pucci(M)$ for each $M\in\Sy$.

\medskip

Furthermore, an equivalent way of stating (H1) and (H2) is to assume $F$ is of the form
\begin{equation}\label{flip}
\displaystyle F(D^2u)=\sup_{\alpha\in \mathcal{A}}\inf_{\beta\in \mathcal{B}} \left(-a^{\alpha,\beta}_{ij}\partial_{ij}u \right) \quad \mbox{ or }\quad F(D^2u)=\inf_{\alpha\in \mathcal{A}}\sup_{\beta\in \mathcal{B}} \left(-a^{\alpha,\beta}_{ij}\partial_{ij}u\right),
\end{equation}
 where $\alpha,\beta$ are indices that belong to some sets $\mathcal{A}$ and $\mathcal{B}$, and the symmetric matrices $A^{\alpha,\beta}=(a^{\alpha,\beta}_{ij})$ satisfy the inequality $\lambda I\le A^{\alpha,\beta}\le\Lambda I$. This is the form of general {\it Isaacs} operators, which are fundamental in the theory of two-player zero-sum stochastic differential games. In the particular case $\mathrm{card}(\mathcal{B})=1$, the operator $F$ in \EQ{flip} is called a Hamilton-Jacobi-Bellman operator -- these have many uses in applied mathematics, and arise in the theory of stochastic optimal control. We refer to Cabre \cite{Cab} as well as Fleming and Soner \cite{FM} for references and more on  Hamilton-Jacobi-Bellman and Isaacs equations.

Suppose that $\Omega$ is an open subset of $\R^n$, $F$ satisfies (H1), and $f\in C(\Omega)$.  A continuous function $u\in C(\Omega)$ is a \emph{viscosity subsolution (resp. supersolution)} of the equation
\begin{equation}\label{defviscsol}
F(D^2u) = f(x) \quad \mbox{in} \ \Omega,
\end{equation}
if, for every point $x_0\in \Omega$ and \emph{test function} $\varphi\in C^2(\Omega)$ such that $x \mapsto u(x) - \varphi(x)$ has a local maximum (resp. minimum) at $x_0$, we have
\begin{equation*}
F(D^2\varphi(x_0)) \leq \ (\mathrm{resp}.\,\geq)\ f(x_0).
\end{equation*}
We say that $u$ is a \emph{viscosity solution} of \EQ{defviscsol} if it is both a viscosity subsolution and supersolution of \EQ{defviscsol}.

Below we mention some standard results from the theory of viscosity solutions which will be used in this article. All differential operators $F$, $G$, $H$, appearing below are assumed to satisfy only (H1).

\begin{itemize}

\item \emph{Maximum principle and strong maximum principle}
(\cite[Proposition 4.9, Theorem 5.3]{CC} and Theorem 3.3 in \cite{UG}
together with the remarks in Example 3.6 and section 5.C in that
paper). Suppose that $\Omega$ is bounded and $u,v\in C(\bar \Omega)$,
$f\in C(\Omega)$ satisfy $F(D^2u) \leq f \leq F(D^2v)$ in $\Omega$ and
$u\leq v$ on $\partial \Omega$ . Then $u \leq v$ in $\Omega$. If
$u(x_0) = v(x_0)$ at some point $x_0\in \Omega$, then $u\equiv v$ in
$\Omega$.

\item \emph{Transitivity of inequalities in the viscosity sense}
(\cite[Theorem 5.3]{CC} and \cite[Lemma 3.6]{A}). Assume  $F(M)+G(N)
\geq H(M+N)$ for each $M,N \in \Sy$. If $u,v,f, g\in C(\Omega)$ are
such that   $F(D^2u) \leq f$ and $G(D^2v) \leq g$ in $\Omega$, then
the function $w:=u+v$ satisfies $H(D^2w) \leq f+g$ in $\Omega$.

\item \emph{Local H\"older estimates} (\cite[Theorem 4.10]{CC}).
Suppose that $u,f\in C(\Omega)$ satisfy the inequalities $\pucci(D^2u)
\leq |f|$ and $\Pucci(D^2u) \geq -|f|$ in $\Omega$. Then there is a
constant $0 < \gamma < 1$, $\gamma=\gamma(n,\lambda,\Lambda)$,  such
that for each compact subset $K \subseteq \Omega$ we can find
$C=C(n,\Elp,\elp,K,\Omega)$ for which
\begin{equation*}
\| u \|_{C^\gamma(K)} \leq C \left( \| u \|_{L^\infty(\Omega)} + \| f
\|_{L^n(\Omega)} \right).
\end{equation*}

\item \emph{The infimum of a family of supersolutions which is
uniformly bounded below is a supersolution.} (\cite[Proposition
2.7]{CC}).

\end{itemize}

The book \cite{CC} is a nice introduction to the theory of viscosity solutions of fully nonlinear, uniformly elliptic equations. See also Crandall, Ishii, and Lions \cite{UG}.

\section{Further remarks and proofs}\label{proofs}
\subsection{Discussion and more general results}

To put our main result and its proof in a proper context, we begin this section with a discussion. First, all earlier results concerning linear and quasilinear operators in divergence form use the weak formulation of the equation in terms of integrals, a feature which operators in non-divergence form do not have. For example, a simple way to prove a nonexistence result for \EQ{lapl} is to replace $u$ by its spherical mean, which is a supersolution too, by Jensen's inequality. It is then enough to study radial supersolutions, which can be reduced to an ODE problem (see for example Guedda and V\'eron \cite{GV}).

 A different approach is required in the nondivergence setting. It was noticed by Labutin \cite{L} as well as in \cite{CL} and \cite{FQ} that for certain rotationally invariant fully nonlinear operators $F$, such as the Pucci extremal operators, there exists a unique  number $\alpha=\alpha^*(F)\in(-1,\infty)$ for which the function $\xi_{\alpha}$ defined by
\begin{equation}\label{eq:radfun}
\xi_\alpha(x) := \left\{ \begin{array}{ll}
|x|^{-\alpha} & \mathrm{if} \ \ \alpha > 0, \\
- \log |x| & \mathrm{if} \ \ \alpha = 0, \\
- |x|^{-\alpha} & \mathrm{if} \ \ \alpha < 0,
\end{array}\right.
\end{equation}
is a smooth solution of $F(D^2\xi_\alpha) = 0$ in $\R^n \setminus \{ 0 \}$. This is easy to verify by a direct computation. For the Pucci maximal  operator $\Pucci$ we have $\alpha^*(\Pucci) = \frac{\Lambda}{\lambda}(n-1)-1$, while the Pucci minimal operator $\pucci$ has scaling exponent $\alpha^*(\pucci) = \frac{\lambda}{\Lambda}(n-1) - 1$.

It is then shown in \cite{CL,FQ} that \EQ{model} has no positive supersolutions in the whole space $\R^n$ if and only if \EQ{critex} holds. The argument used in these papers  relies on the fact  that any solution of $F(D^2u) \ge 0$ in $\R^n$ satisfies $\min_{\partial B_r} u=\min_{\bar B_r} u$ by the maximum principle, so the function  $r\mapsto m(r) = \min_{\partial B_r} u$ is decreasing. Together with the radial symmetry of the fundamental solution, this permits one to prove an analogue of Hadamard's three spheres theorem, which in turn implies that the function $r^{\alpha^*(F)}m(r)$ is increasing. Finally, with the help of a judiciously chosen test function, one obtains  the inequality $m(r)^p\le Cr^2 m(r/2)$, which contradicts properties of $m(r)$. This approach cannot be used if the operator is not rotationally invariant, or if the domain is not $\R^n$.

In addition to being more general, our proof of the corresponding Liouville-type result, Theorem \ref{theomodel}(i), is actually simpler. It has two central ideas. First, we observe that a nontrivial solution of $F(D^2u) \ge 0$ in an exterior domain must be greater than a constant multiple of the fundamental solution, according to the maximum principle. Second, thanks to the scaling properties of \EQ{model}, we will see that in the case $\alpha^*(F) \leq \beta^*(p)$ the existence of a nontrivial solution of \EQ{model} contradicts the finiteness of the first half-eigenvalue of $F$ in an annular domain.

The existence statement (ii) in Theorem \ref{theomodel} is obtained by ``bending" the fundamental solution, that is, by showing that some power of $\Phi$ is a supersolution in $\R^n\setminus\{0\}$ of the inequality we want to solve. Then a supersolution in $\R^n$ is obtained by truncating this function around the origin in a suitable way.

The proof of Theorem \ref{theomodel}(iii) is based on a well-known theorem by Krasnoselskii, which asserts the existence of a nonzero fixed point of any compact map which sends a convex cone into itself, under some conditions on its behavior on two distinct spheres. The precise statement of this theorem is given in Theorem \ref{fixedpt} below. To apply it, we show that an appropriate map can be defined on the cone of nonnegative $(-\beta^*)$-homogeneous functions whenever $\beta^*<\alpha^*$. We verify that its fixed points are solutions of \EQ{model}, and that it satisfies the conditions of Krasnoselskii's theorem under the assumption that \EQ{model} does not have bounded positive solutions in the whole space $\R^n$. On the other hand, if \EQ{model} has a positive solution in $\R^n$, then of course this solution is also a solution in $\R^n \setminus \{ 0 \}$, so we have nothing to prove.

\smallskip

As we noted in the introduction, we can easily extend our results to more general nonlinearities $f(x,u)$.
The next theorem generalizes our statements on supersolutions, that is, Theorem \ref{theomodel} (i) and (ii). We denote
\begin{equation*}
\beta^*=\beta^*(p,\gamma):=\frac{2-\gamma}{p-1},
\end{equation*}
for each $\gamma <2$ and $p > 1$.

\begin{thm}\label{mainthm}
Assume that the operator $F$ satisfies (H1) and (H2), $R_0> 0$,  and $f: (\R^n \setminus B_{R_0}) \times (0,\infty) \to (0,\infty)$ is a continuous function (note $f$ needs not be defined at $u=0$).
 \begin{enumerate}
 \item Suppose there exist $\varepsilon_0, c_0, C_0>0$, $p>1$, $\gamma<2$ such that $\alpha^*(F)\le \beta^*(p,\gamma)$,
 \begin{equation}\label{fx-nonexist1}
 f(x,s)\geq c_0|x|^{-\gamma}s^p\quad\mbox{in}\ (\R^n\setminus B_{R_0})\times(0,\varepsilon_0), \ \mbox{and}
 \end{equation}
 \begin{equation}\label{fx-nonexist2}
\frac{f(x,s)}{s}\le C_0 \frac{f(x,t)}{t} \quad\mbox{for all} \ t>0, \ s\in(0,\min\{t,\varepsilon_0\}), \ |x|\ge R_0.
\end{equation}
Then \EQ{pdiegen} has no positive supersolution in any exterior domain.

 \item Suppose that there exist constants $\varepsilon_0,C_0 >0$, $p>1$ and $\gamma < 2$ such that $\alpha^*(F) > \beta^*(p,\gamma)$ and
\begin{equation}\label{fx-exist}
 f(x,s)\le C_0|x|^{-\gamma}s^p\quad\mbox{in}\ (\R^n\setminus B_{R_0})\times(0,\varepsilon_0).\end{equation}
Then there exists a positive supersolution of \EQ{pdiegen} in $\R^n\setminus B_{R_0}$. If in addition $\gamma \le 0$ and \EQ{fx-exist} holds in $\R^n\times(0,\varepsilon_0)$, then \EQ{pdiegen} has a positive supersolution in the whole space $\R^n$.
\end{enumerate}
\end{thm}

\begin{remark}\label{remk} The number $\beta^*(p,\gamma)$ defined in this theorem is the scaling exponent of the equation
\begin{equation} \label{pdie}
F(D^2u) = |x|^{-\gamma} u^p.
\end{equation}
That is, if $u$ is a (sub/super)solution of \EQ{pdie} in $\R^n \setminus B_R$, then the rescaled function
\begin{equation}\label{rescale}
u_\sigma (x):= \sigma^{\beta^*} u(\sigma x), \qquad \beta^*=\frac{2-\gamma}{p-1}\,,
\end{equation}
 is also a (sub/super)solution of \EQ{pdie} in $\R^n \setminus B_{R/\sigma}$.
\end{remark}

\begin{remark}
In the case $f = f(s)$ does not depend on $x$, conditions \EQ{fx-nonexist1} and \EQ{fx-exist} in Theorem \ref{mainthm} suggest that it is the behaviour of $f(s)$ near $s=0$ that determines whether there exist positive supersolutions of \EQ{pdiegen} in exterior domains. While condition \EQ{fx-nonexist2} is a global hypothesis, it is not very restrictive. In fact, hypothesis \EQ{fx-nonexist2} turns out to be necessary only for the method of proof we employ. We have recently discovered another approach which permits us to replace \EQ{fx-nonexist2} by an optimal hypothesis, and yields nonexistence results for sublinear equations ($p<1$) as well as for systems of Lane-Emden type. This alternative method, based on some results from the regularity theory of Krylov and Safonov (c.f. \cite{K}), will be described in a forthcoming article.
\end{remark}

\subsection{Proof of the nonexistence result}
The existence of fundamental solutions of quasilinear equations provides a lower bound on positive supersolutions, as observed for example in \cite[Lemma 2.3] {SZ} and \cite[Proposition 2.6]{BVP}. The following  lemma states that the same is valid for uniformly elliptic operators of Isaacs type. Its simple proof needs only Theorem~\ref{fundysol} and the maximum principle.

\begin{lem}\label{lower} Suppose that $R_1>0$ and $u\in C(\R^n \setminus B_{R_1})$ is a positive solution of  $F(D^2u)\ge 0$ in $\R^n\setminus \bar B_{R_1}$. Then for some $c>0$,
\begin{equation}\label{lowerb}
u(x)\ge c|x|^{-\alpha}\;\;\mbox{ in }\;\R^n\setminus B_{R_1},\qquad \mbox{where}\quad \alpha : = \max\{0,\alpha^*(F)\}.
\end{equation}
\end{lem}
\begin{proof} We first consider the case $\anom(F)>0$. By Theorem \ref{fundysol}, the fundamental solution $\Phi$ of $F$ is such that $\Phi>0$ and $\Phi(x)\to 0$ as $|x|\to \infty$. Select $c>0$ so small that $u\ge c\Phi$ on $\partial B_{R_1}$. Then for each $\varepsilon >0$, there exists $\bar R= \bar R(\varepsilon) > R_1$ such that $u+\varepsilon\ge \varepsilon\ge\Phi$ in $\R^n\setminus B_{\bar R}$. Applying the maximum principle to
$$
F(D^2(u+\varepsilon))\ge 0= F(D^2\Phi)
$$
in $B_{R}\setminus B_{R_1}$, for each $R>\bar R(\ep)$, we conclude that $u+\varepsilon\ge \Phi $ in $\R^n\setminus B_{R_1}$. Letting $\varepsilon\to 0$ we obtain $u\ge \Phi$ in $\R^n\setminus B_{R_1}$. This implies \EQ{lowerb} by the homogeneity of $\Phi$, since $\Phi(x)= |x|^{-\alpha^*(F)}\Phi(x/|x|)$ for every $x\in \R^n \setminus \{ 0 \}$.

In the case $\anom(F)\le 0$, we have $\Phi(x) \to -\infty$ as $|x|\to \infty$. Setting $$c:=(1/2)\min_{\partial B_{R_1}}u>0,$$ we observe that $u\ge \varepsilon \Phi + c$ on $\partial(B_R\setminus B_{R_1})$, for each $0<\varepsilon<c/(\max_{\partial B_{R_1}} \Phi)$, and each $R>\bar R(\ep)$ sufficiently large. By the maximum principle, we have $u\ge \varepsilon \Phi + c$ in $\R^n\setminus B_{R_1}$. By passing to the limit $\varepsilon\to 0$, we deduce that $u\ge c$.
\end{proof}

\begin{proof}[Proof of \THM{mainthm} (i)]
Assume first that there exists $u>0$ which is a supersolution of \EQ{pdiegen} with $f(x,u)=|x|^{-\gamma} u^p$, that is,
\begin{equation}\label{pdieagain}
F(D^2u ) \geq |x|^{-\gamma} u^p
\end{equation}
in some exterior domain $R^n\setminus B_{R_1}$. We can assume $R_1 > 1$. We will argue that
the number $\beta^* $ defined in \EQ{rescale} is such that $\beta^*<\alpha^*(F)$.

In what follows $c,C$ denote positive constants which may change from line to line, and depend only on $F$ and $n$.

 By Remark \ref{remk}, for each $\sigma\ge R_1$ the function $u_\sigma$ given by \EQ{rescale} is a supersolution of \EQ{pdieagain} in $\R^n \setminus B_{1}$, and hence
 \begin{equation}\label{pneverdie}
 F(D^2u_\sigma)\ge |x|^{-\gamma} u_\sigma^p \ge \min\{1,2^{-\gamma}\}  \, u_\sigma^p\quad\mbox{ in }\; B_2\setminus B_{1}.
 \end{equation}

By applying Lemma \ref{lower} to $u$ and expressing \EQ{lowerb} in terms of $u_\sigma$, we obtain
\begin{equation}\label{usigbel}
u_\sigma(x)=\sigma^{\beta^*}u(\sigma x) \geq c \sigma^{\beta^* - \alpha} |x|^{-\alpha} \quad \mbox{for every} \ \sigma \geq 1 \ \mbox{and} \ |x| \geq R_1/\sigma,
\end{equation}
where $\alpha := \max\{0,\alpha^*(F)\}.$ Thus for every $\sigma \geq R_1,$ we have $u_\sigma(x) \geq c \sigma^{\beta^* - \alpha}$ in $B_2\setminus B_{1}$. Hence by \EQ{pneverdie},
\begin{equation}\label{inqi}
F(D^2 u_\sigma) \geq  c \sigma^{(\beta^*-\alpha)(p-1)}\,u_\sigma \quad \mbox{in} \ B_2\setminus B_{1}.
\end{equation}
The existence of a positive function $u_\sigma$ satisfying this inequality implies that the \emph{first eigenvalue} of $F$ in the bounded regular domain $B_2\setminus B_{1}$ is bounded below by $c \sigma^{(\beta^*-\alpha)(p-1)}$, for every $\sigma \geq R_1$. This is of course a contradiction if $\beta^*> \alpha$ and $\sigma \geq 2R_1$ is taken large enough, since the first eigenvalue $\lambda^+_1(F,B_2\setminus B_1)$ is finite. For a simple proof\footnote{In view of possible applications of this method to other problems, we note that this argument does not depend on the existence of an eigenfunction.} of a finite upper bound on the first eigenvalue, which requires only the maximum principle, see the definition of $\lambda^+_1(F,\Omega)$ as well as Lemma 3.7 in \cite{A}. For more on first eigenvalues of fully nonlinear equations of Isaacs type, see \cite{QS2,A} and the references therein.

\medskip

We have left to rule out the critical case $ \alpha= \alpha^*(F)=\beta^* >0 $. We first show that in this case we can improve \EQ{usigbel}. Define $w(x): = \Phi(x) \left( \log |x| \right)$ and observe that by (H1) and (H2),
\begin{eqnarray*}
F(D^2w)&\le&  \log|x|F(D^2\Phi) +2\Pucci(D\log|x|\otimes D\Phi) + \Phi\Pucci(D^2\log|x|)\\
&=&  2 |x|^{-2} \Pucci\left( x \otimes D\Phi(x) \right) + \Phi(x)\Pucci\left(  |x|^{-2}\iden  - 2|x|^{-4} x\otimes x \right)
\end{eqnarray*}
in $\R^n\setminus B_{R_1}$, provided that $\Phi \in C^2(\R^n\setminus B_{R_1})$. We remark that differentiating the equality ${\Phi}(x) =t^{\alpha} {\Phi}(tx)$ with respect to $x$  yields that the matrix $H(x) = \frac{x\otimes D\Phi(x)}{\Phi(x)}$ satisfies $H(tx) = H(x)$ for all $t>0$. Thus $\Pucci(H(x))\le C=\max_{\partial B_2}\Pucci(H(x))$ in $\R^n\setminus B_1$. The latter is rigorous, since $\Phi$ is locally uniformly bounded in $C^1$, depending only on the constants $\lambda$, $\Lambda$, and $n$. Thus  $\Pucci\left( x \otimes D\Phi(x) \right)\le C \Phi(x)$, and since the eigenvalues of $x\otimes x $ are $0$ and $|x|^2$,  we obtain
\begin{equation}\label{ineqw}
F(D^2w) \leq  C |x|^{-2} \Phi(x)  \leq C |x|^{-\alpha-2} \quad \mbox{in} \ \R^n \setminus B_{R_1}.
\end{equation}
By performing an analogous calculation with a smooth test function, we confirm that the differential inequality \EQ{ineqw} holds rigorously in the viscosity sense, even when $\Phi$ is only in $C^{1,\gamma}_{\mathrm{loc}}$ but not in $C^2$ (see for instance the proof of Lemma 3.3 in \cite{ASS}). According to Lemma \ref{lower} and \EQ{pdieagain}, we also have
\begin{equation}\label{inequ}
F(D^2u) \geq |x|^{-\gamma} u^p \geq c|x|^{-\gamma-p\alpha}=c |x|^{-\alpha-2} \quad \mbox{in} \ \R^n \setminus B_{R_1},
\end{equation}
since $\alpha=\beta^*$. Note that Theorem \ref{fundysol} implies $w=\Phi(x)\log |x|\to 0$ as $|x|\to \infty$, since $\alpha^*=\beta^*>0$.
By \EQ{ineqw}, \EQ{inequ}, and the maximum principle, we infer that
\begin{equation*}
u +\ep \geq c w   \quad \mbox{in} \ B_R \setminus B_{R_1}
\end{equation*}
for every $\ep > 0$ and every $R>\bar R(\ep)$ sufficiently large. Therefore by letting first $R\to \infty$ and then $\ep\to0$ we obtain
\begin{equation*}
u \geq c |x|^{-\alpha} \log |x| \quad \mbox{in} \ \R^n \setminus  B_{R_1},
\end{equation*}
by the definition of $w$ and Theorem \ref{fundysol}. Rescaling, we find that for $\sigma \geq R_1$,
\begin{equation*}
u_\sigma \geq c \log \sigma \quad \mbox{in} \ B_2 \setminus B_1.
\end{equation*}
Thus for all $\sigma \geq R_1$,
\begin{equation*}
F(D^2 u_\sigma) \geq |x|^{-\gamma} u_{\sigma}^p \geq c (\log \sigma)^{p-1} u_\sigma \quad \mbox{in} \ B_2 \setminus B_1.
\end{equation*}
As above, this shows that the first eigenvalue of $F$ in $B_2\setminus B_1$ is bounded below by $c_3 (\log \sigma)^{p-1}$.
Sending $\sigma \to \infty$ yields a contradiction. This completes the proof of Theorem \ref{mainthm}(i) in the case that $f(x,u)=|x|^{-\gamma} u^p$.

Finally, we remark that in the case of general nonlinearity $f$ satisfying the hypotheses of Theorem \ref{mainthm} the problem can be reduced to the particular case we have just studied. Namely, with the help of Lemma \ref{lower} we obtain
\begin{equation}\label{endi}
\frac{f(x,u)}{u^p}\ge \frac{f(x,c|x|^{-\alpha})}{u^{p-1}c|x|^{-\alpha}}\ge \frac{c|x|^{-\gamma -\alpha(p-1)}}{u^{p-1}} = c|x|^{-\gamma}\left( \frac{c|x|^{-\alpha}}{u}\right)^{p-1}\ge c|x|^{-\gamma}
\end{equation}
in $\R^n\setminus B_{R_1}$. The first inequality in \EQ{endi} follows from Lemma \ref{lower} and \EQ{fx-nonexist2} (if necessary, we take larger $R_1$, so that $R_1\ge R_0$ and $c|x|^{-\alpha}\le \varepsilon_0$ in $\R^n\setminus B_{R_1}$), the second inequality is a consequence of \EQ{fx-nonexist1}, and the last inequality follows again from Lemma \ref{lower}. Hence the nonexistence statement in Theorem \ref{mainthm} follows from what we already proved.

\subsection{Proof of \THM{mainthm} (ii)} We suppose $0<(2-\gamma)/(p-1)=\beta^* < \alpha^*(F)$. Define $v:= \Phi^{\tau}$ where $\tau : = \beta^* /\alpha^*(F) \in (0,1)$. In the case $\Phi\in C^2$, we have
\begin{equation*}
D^2v = \tau \Phi^{\tau -1} D^2 \Phi - \tau(1-\tau) \Phi^{\tau -2} \left( D\Phi \otimes D\Phi \right),
\end{equation*}
 so by (H1) and (H2) we obtain
\begin{align*}
F(D^2v) & \geq \tau \Phi^{\tau -1} F(D^2 \Phi) + \tau(1-\tau) \Phi^{\tau-2}\, \pucci \!\left( - D\Phi \otimes D\Phi \right) \\
& = \lambda \tau (1-\tau) \Phi^{\tau -2} |D\Phi|^2
\end{align*}
in $\R^n \setminus \{ 0 \}$. This is routine to confirm in the viscosity sense in the case $\Phi \not\in C^2$.

By differentiating $\Phi(x)=\sigma^\alpha \Phi(\sigma x)$ with respect to $\sigma$, we see that $x\cdot D\Phi = -\alpha \Phi$, hence $|D\Phi| \geq \alpha |x|^{-1} \Phi$.
Therefore,
\begin{equation*}
F(D^2v) \geq c|x|^{-2}\Phi^\tau \geq c |x|^{-\tau \alpha^* - 2} \quad \mbox{in} \ \R^n \setminus \{ 0 \}.
\end{equation*}
Since $0 < c|x|^{-\beta^*}\le v(x) \leq C|x|^{-\beta^*}$ and $-\beta^* p - \gamma = -\tau \alpha^* -2$, we deduce that
\begin{equation}\label{ineqexv}
F(D^2v) \geq c |x|^{-\gamma} v^p \quad \mbox{in} \ \R^n \setminus \{ 0 \}.
\end{equation}

Suppose in addition that $\gamma \leq 0$. For $a > 0$, let $w \in C(\bar B_1)$ be the unique solution of the Dirichlet problem
\begin{equation*}
\left\{ \begin{aligned}
& F(D^2w) = a & \mbox{in} & \ B_1, \\
& w = 0 & \mbox{on} & \ \partial B_1.
\end{aligned} \right.
\end{equation*}
The ABP inequality (\cite[Theorem 3.6]{CC}) provides us the estimate $0 < w \leq Ca$. Thus if $a > 0$ is sufficiently small, the function $w$ satisfies
\begin{equation}\label{ineqexw}
F(D^2w) \geq |x|^{-\gamma} w^p \quad \mbox{in} \ B_1.
\end{equation}
By multiplying $v = \Phi^\tau$ by a small constant, if necessary, we may assume without loss of generality that $v$ satisfies \EQ{ineqexv} and $v < w$ in $B_{1/2} \setminus B_{1/4}$. Since $v(x) \to \infty$ as $|x| \to 0$, there exists $\delta > 0$ such that $w < v$ in $B_{2\delta}$. Now define the function
\begin{equation*}
u(x) := \begin{cases}
w(x) & x\in B_\delta, \\
\min \{ v(x), w(x) \} & x \in B_{1/3} \setminus B_\delta, \\
v(x) & x \in \R^n \setminus B_{1/3}.
\end{cases}
\end{equation*}
By construction, $u$ is a supersolution of \EQ{pdie} in the whole space $\R^n$, since the minimum of two supersolutions is also a supersolution. This completes the proof of Theorem \ref{mainthm} (ii), in the special case $f(x,u)=|x|^{-\gamma} u^p$.

\smallskip

Moreover, inequalities \EQ{ineqexv} and \EQ{ineqexw} continue to hold if we replace $v$ (resp. $w$) by $av$ (resp. $aw$), for  any $a\le 1$.  Since we can find $a$ so small that $av\le \varepsilon_0$ in $\R^n \setminus \{ B_{1/3} \}$ (resp. $aw\le \varepsilon_0$ in $B_1$),  Theorem~\ref{mainthm} (ii) follows from the existence result we just established.
\end{proof}

\subsection{Proof of Theorem \ref{theoexist}}

It is convenient to define the space
\begin{equation*}
X_\alpha := \left\{ u \in C(\R^n \setminus \{ 0 \}) : \ u(x) = \sigma^{\alpha} u(\sigma x) \ \mbox{for every} \ x\in \R^n \setminus \{ 0 \}, \ \sigma > 0\right\}
\end{equation*}
for every $\alpha > 0$, as well as
\begin{equation*}
H_\alpha := \left\{ x \in X_\alpha: u \geq 0 \right\} \quad \mbox{and} \quad H^+_\alpha := \left\{ u \in X_\alpha :  u > 0 \right\}.
\end{equation*}
Observe that $X_\alpha$ is a Banach space under the norm $\| u \|_{X_\alpha} = \max_{\partial B_1} |u|$, and $H_\alpha$ is a closed convex cone in $X_\alpha$ with interior $H^+_\alpha$.

Our argument is based on the following well-known fixed point theorem, due to Krasnoselskii (see \cite{KZ}). It also appears in the appendix of Benjamin \cite{Be}, and is applied to semilinear elliptic equations in \cite{DLN}.

\begin{thm}[Krasnoselskii] \label{fixedpt}
Let $X$ be a Banach space, and $C$ a closed convex cone in $X$ with vertex at the origin. Consider a compact map $A:C\to C$ which satisfies $A(0) = 0$. Suppose there exist $0< \bar r < \bar R$ and $\xi \in C\setminus \{ 0 \}$ such that:
\begin{enumerate}
\item $u \neq t A(u)$ for every $0 \leq t \leq 1$ and $\| u \|_X = \bar r$, and
\item $u \neq A(u) + t\xi$ for every $t \geq 0$ and $\| u \|_X = \bar R$.
\end{enumerate}
Then there exists $u\in C$ satisfying $A(u) = u$ and $\bar r < \| u \|_X < \bar R$.
\end{thm}

In this subsection we set $\beta=\beta^*(p)$, where $p>1$ is fixed such that $\beta < \alpha^*(F)$. We are going to apply Theorem \ref{fixedpt} using the nonlinear map $A:H_\beta \to H_\beta$ which is defined for each $v\in H_\beta$ by $A(v) := u$, where $u \in H_\beta$ is the unique solution of the equation
\begin{equation*}
F(D^2u) = v^p \quad \mbox{in} \ \R^n \setminus \{ 0 \}.
\end{equation*}
That $A$ is well-defined is a consequence of \cite[Lemma 3.8]{ASS}, our assumption that $0<\beta < \alpha^*(F)$, and the relation between $\beta$ and $p$. Precisely, the number $\beta$ has the property that if $\beta p = \beta +2$, and therefore if $v\in H_\beta$, then $v^p \in H_{\beta + 2}$. Now \cite[Lemma 3.8]{ASS} asserts that if $0 < \alpha < \alpha^*(F)$ and $f\in H_{\alpha}$, then the equation
\begin{equation*}
F(D^2u) = f \quad \mbox{in} \ \R^n \setminus \{ 0 \}
\end{equation*}
has a unique solution $u \in H_\alpha$. It follows that $A:H_\beta \to H_\beta$ is well-defined.

Clearly $A(0)=0$, while the strong maximum principle implies that $A(v) \in H^+_\alpha$ for every $v\in H_\alpha \setminus \{ 0 \}$. By the H\"older regularity result which we quoted in the previous section, and since viscosity solutions are stable under local uniform convergence (see for example \cite[Theorem 3.8]{CCKS}), the map $A$ is compact and continuous.

In the next lemma we verify that hypothesis (i) in Theorem \ref{fixedpt} holds for $A$, that is, we find our inner radius $\bar r>0$.

\begin{lem} \label{below}
Suppose that  $u \in H_\beta$, $u\not \equiv0$, is a solution of  $u = tA(u)$ for some $0 \leq t \leq 1$. Then there exists a constant $\bar r>0$, which does not depend on $u$, such that
\begin{equation} \label{estbelow}
\| u \|_{X_\beta} > \bar r.
\end{equation}
\end{lem}
\begin{proof}
The equation $u = tA(u)$ means that $u\in H_\beta$ is a solution of the equation
\begin{equation*}
F(D^2u) = tu^p \quad \mbox{in} \ \R^n \setminus \{ 0 \}.
\end{equation*}
Note that $t=0$, that is, $F(D^2u)=0$ in $\R^n \setminus \{ 0 \}$ is excluded by Theorem \ref{fundysol} and our assumptions $u\not\equiv 0$ and $\beta <  \alpha^*(F)$. Hence $u\in H_\beta^+$ by the strong maximum principle.

By \cite[Lemma 3.8]{ASS}  there exists a function $v \in H_\beta^+$ satisfying
\begin{equation*}
F(D^2v) = |x|^{-\beta-2} \quad \mbox{in} \ \R^n \setminus \{ 0 \}.
\end{equation*}
Since $-\beta -2 = -\beta p$, we may set $w: = av$ for some  small $0 < a < 1$ to discover that
\begin{equation*}
F(D^2w) \geq 2 w^p \quad \mbox{in} \ \R^n \setminus \{ 0 \}.
\end{equation*}
Set $\bar r:= \min_{\partial B_1} w$ and suppose that \EQ{estbelow} fails. Then $w \geq u$, and so by multiplying $w$ by a positive constant at most 1, we may assume that $\min_{\partial B_1} (w-u) = 0$, which by the homogeneity of $u$ and $w$ means that $\min_{\R^n \setminus \{ 0 \}} (w-u) = 0$. By the strong maximum principle, we deduce that $w \equiv u$, which is impossible since $t\leq 1$.
\end{proof}

In the next lemma, we find the outer radius $\bar R>0$ in Theorem \ref{fixedpt}. Here we set $$\xi = \xi_\beta= |x|^{-\beta}\in H_\beta^+.$$

\begin{lem} \label{above}
Assume that equation \EQ{model} has no bounded positive solutions. Suppose that  $u \in H_\beta$ satisfies $u = A(u) + t\xi$ for some $t\geq 0$. Then there exists a constant $\bar R>0$, which does not depend on $u$ and $t$, such that
\begin{equation} \label{estabove}
\| u \|_{X_\beta} < \bar R.
\end{equation}
\end{lem}

\begin{proof} We can assume that $u\not\equiv 0$.
The equation
\begin{equation} \label{upeq}
u = A(u) + t \xi
\end{equation}
means that the function $w:= u - t\xi$ is in  $H_\beta$ and satisfies the equation
\begin{equation*}
F(D^2w) = u^p \quad \mbox{in} \ \R^n \setminus \{ 0 \}.
\end{equation*}
In particular $w=u- t\xi\ge 0$. By the strong maximum principle  $w\in H_\beta^+$ and hence $u\in H_\beta^+$.

Put $m := \min_{\partial B_1} u>0$. We first claim that there exists $T> 0$, which does not depend on $u$ and $t$, such that $t\le m\leq T$. Notice $u\ge t\xi\ge 0$ implies $t\le m$.

Recall $\beta p = \beta+2$. Let $v\in H_\beta^+$ be the solution of $$ F(D^2 v) = \xi^p.$$ Set $a=\min_{\partial B_1} v>0$ and notice that $a$ depends only on $F$ and $n$. Then we have
$$
F(D^2 w ) = u^p \geq m^p \xi^p = F(m^pv)
$$
in $\R^n \setminus \{ 0 \}$, so that \cite[Proposition 3.2(iii)]{ASS} implies $w \geq m^p v$ in $\R^n \setminus \{ 0 \}$. In particular,
\begin{equation*}
u - t = w \geq a m^p \quad \mbox{on} \ \partial B_1.
\end{equation*}
Thus
\begin{equation*}
m \geq m - t \geq am^p > 0.
\end{equation*}
Hence $t \leq m\le a^{-1/(p-1)}=:T$, and we have the claim.

We now prove the bound \EQ{estabove} by contradiction. If it fails, then for each $k\geq 1$ there exist $0 \leq t_k \leq T$ and $u_k \in H_\beta^+$ such that the function $w_k:= u_k - t_k \xi$ satisfies
\begin{equation*}
F(D^2w_k) = u_k^p\quad \mbox{in} \ \R^n \setminus \{ 0 \},
\end{equation*}
but
\begin{equation*}
\| u_k \|_{X_\beta} = \max_{\partial B_1} u_k\geq k.
\end{equation*}
For each $k$, select $x_k \in \partial B_1$ such that $u_k(x_k) = \| u_k \|_{X_{\beta}}$. Observe that by the  homogeneity of $u$,
\begin{equation} \label{attnuk}
u_k(y_k) = 1, \qquad y_k : =  \| u_k \|_{X_{\beta}}^{1/{\beta}} x_k.
\end{equation}
Moreover,
\begin{equation} \label{bounduk}
u_k \leq 2 \quad \mbox{in} \ B_{R_{k}}(y_k), \quad R_{k}: =  \| u_k \|_{X_{\beta}}^{1/{\beta}} \left( 1 - 2^{-1/{\beta}} \right).
\end{equation}
Define $v_k (z) : = u_k (y_k + z)$, and observe that $v_k$ satisfies $v_k(0) = 1$, $0\leq v_k \leq 2$ in $B_{R_{k}}(0)$, and the equation
\begin{equation} \label{keq}
F(D^2v_k(z) - t_k D^2\xi(y_k+z)) = v_k^p \quad \mbox{in} \ \R^n \setminus \{ -y_k\} \supseteq B_{R_{k}}(0).
\end{equation}
Observe that
\begin{equation*}
D^2\xi(z) = \beta(\beta+2)|z|^{-\beta-4} z\otimes z - \beta |z|^{-\beta-2}\iden.
\end{equation*}
Since $|y_k| \to \infty$ and $0\leq t_k \leq T$, we deduce that $t_k D^2\xi(y_k+z)\rightarrow 0$ locally uniformly in $z\in \R^n$ as $k \to \infty$. We also have $R_{k} \to \infty$ as $k\to \infty$.

We will now argue that we may pass to limits in \EQ{keq} in order to  find a function $v$ satisfying $0 \leq v \leq 2$, $v(0)=1$, and the equation
\begin{equation*}
F(D^2v) = v^p \quad \mbox{in} \ \R^n.
\end{equation*}
First, observe that using \EQ{keq} and (H1), we have that for $|z| \leq R_k$,
\begin{align*}
\pucci(D^2v_k(z)) & \leq F( D^2v_k(z) - t_k D^2\xi(y_k+z)) - F(-t_k D^2\xi(y_k+z)) \\
& \leq (v_k(z))^p + C T |y_k+z|^{-\beta -2} \\
& \leq (v_k(z))^p + CT\left( |y_k| - |z| \right)^{-\beta-2} \\
& \leq 2^p + C T \left( \frac{1}{2} \| u_k \|_{X_\beta} \right)^{-(\beta +2)/\beta}
\end{align*}
using also that $|y_k| = \| u_k \|^{1/\beta}_{X_\beta}$. Similarly, we have
\begin{equation*}
\Pucci(D^2v_k) \geq - C T \left( \frac{1}{2} \| u_k \|_{X_\beta} \right)^{-(\beta +2)/\beta} \quad \mbox{in} \ B_{R_{k}}(0).
\end{equation*}
Applying local H\"older estimates quoted in the previous section and using that $\|u_k\|_{X_\beta} \to \infty$ and $R_k \to \infty$ as $k \to \infty$, we have the bound
\begin{equation*}
\sup_{k\geq 1} \| v_k \|_{C^\alpha(B_R(0))} \leq C(R) < \infty
\end{equation*}
for every $R>0$. Therefore, by passing to a subsequence we may assume that $v_k \rightarrow v \in C(\R^n)$ locally uniformly in $\R^n$. It is clear from \EQ{attnuk} and \EQ{bounduk} that $0 \leq v\leq 2$ and $v(0)=1$. If we define the operator
\begin{equation*}
F_k(M,z) := F( M - t_k D^2\xi(y_k+z) ),
\end{equation*}
then $F_k(M,z) \rightarrow F(M)$ locally uniformly in $\Sy \times \R^n$ as $k\to \infty$, since the matrix $t_k D^2\xi (y_k+z)$ tends to $0$ locally uniformly in $z\in\R^n$. Thus we may use \cite[Theorem 3.8]{CCKS} to pass to limits in \EQ{keq} and obtain
\begin{equation*}
F(D^2v) = v^p \quad \mbox{in} \ \R^n.
\end{equation*}
Since $v(0) = 1$, the strong maximum principle implies that $v> 0$ in $\R^n$. We have obtained our desired contradiction, since we had assumed in our hypothesis that this equation has no bounded positive solutions.
\end{proof}

\begin{proof}[Proof of Theorem \ref{theoexist}]
If equation \EQ{model} has a no bounded positive solution in $\R^n$, then Lemmas \ref{below} and \ref{above} assert that the map $A:H_\beta \to H_\beta$ satisfies the hypotheses of \THM{fixedpt}. Therefore there exists $u\in H_\alpha\setminus \{ 0 \}$ satisfying $A(u) = u$, that is, $u$ is a solution of \EQ{model} in $\R^n \setminus \{ 0 \}$. The strong maximum principle implies $u >0$.
\end{proof}

\bibliographystyle{plain}
\bibliography{critsol}

\end{document}